\documentclass[a4paper,11pt]{article}
\usepackage[utf8]{inputenc}
\usepackage[a4paper,margin=2.5cm]{geometry}
\usepackage[utf8]{inputenc}
\usepackage{amsmath,amsthm,amsfonts,amssymb}
\usepackage{hyperref}
\usepackage{mathtools}
\usepackage{todonotes}
\usepackage{cite}
\usepackage{enumerate}
\usepackage{mathrsfs}
\usepackage{bm}
\usepackage{setspace}
\usepackage{aligned-overset}
\usepackage{pgf,tikz}
\usepackage{mathrsfs}
\usetikzlibrary{arrows}

\def\R{\mathbb{R}}

\def\S{\mathbb{S}}
  
\def\F{\mathcal{F}}

\newcommand{\E}{\mathbb{E}}

\newcommand{\Exp}{\mathrm{exp}}
\newcommand\numberthis{\addtocounter{equation}{1}\tag{\theequation}}

\newcommand{\norm}[1]{\left\lVert#1\right\rVert_{\infty}}

\def\Hess{\mathrm{Hess}}
\def\Ric{\mathrm{Ric}}

\theoremstyle{plain}
\newtheorem{theorem}{Theorem}[section]
\newtheorem{corollary}[theorem]{Corollary}
\newtheorem{lemma}[theorem]{Lemma}
\newtheorem{proposition}[theorem]{Proposition}

\theoremstyle{definition}
\newtheorem{definition}[theorem]{Definition}
\newtheorem{example}[theorem]{Example}
\newtheorem*{remark}{Remark}

\title{\textbf{Stein's Method for Probability Distributions on $\mathbb{S}^1$}}
\author{Alexander Lewis\thanks{School of Mathematical Sciences. E-mail: alexander.lewis1@nottingham.ac.uk} \\ \textit{University of Nottingham}}
\date{}

\begin{document}

\maketitle

\begin{abstract}
In this paper, we propose a modification to the density approach to Stein's method for intervals for the unit circle $\mathbb{S}^1$ which is motivated by the differing geometry of $\mathbb{S}^1$ to Euclidean space. We provide an upper bound to the Wasserstein metric for circular distributions and exhibit a variety of different bounds between distributions; particularly, between the von-Mises and wrapped normal distributions, and the wrapped normal and wrapped Cauchy distributions.
\end{abstract}

\section{Introduction}

Since its inception in 1972 \cite{stein1972bound}, Stein's method has been utilised in many research areas to provide a foundation for distributional comparison and approximation. The main objective of such a problem is to bound the (pseudo-) metric
\[ d_{\mathcal{H}}(P,Q):= \sup_{h \in \mathcal{H}}\bigg| \int h dP - \int h dQ \bigg| \]
between an arbitrary probability measure $Q$ and target probability measure $P$, both defined on the state space $\Omega$ with respect to some class of test functions $\mathcal{H}$ on $\Omega$. Stein's approach to this problem has three critical steps:
\begin{enumerate}
\item Find an operator $L$ such that, for a random variable $X \sim P$, $\E[L f(X)]=0$ for a class of functions $f \in \F$ with $f: \Omega \rightarrow \R$.
\item Formulate and solve the so-called Stein equation
\[ L f_h(x) = h(x) - \E[h(X)]. \]
\item Attempt to bound the left hand side of the Stein equation using the vast number of tools that the literature has to offer. This includes, but not limited to, exchangeable pairs, size-biasing, zero-biasing, and operator comparison.
\end{enumerate}
Depending on the approach taken, step 2 may not involve solving the Stein equation.
We point the reader towards the exposition \cite{ross2011fundamentals} which provides a comprehensive guide of Stein's method.

More recently, Stein's method has greatly expanded in scope to many distributional types; general univariate \cite{ley2017stein}, multivariate \cite{mackey2016multivariate,mijoule2018stein}, and manifold valued \cite{le2020diffusion,thompson2020approximation}.
One can split Stein's method into two distinct approaches: the classical Stein density approach, and the diffusion approach. In short, the density approach constructs the operator from the density of the target distribution itself, whereas the diffusion approach uses the infinitesimal generator of an over-damped Langevin diffusion.

The focus of this paper is to motivate and develop a Stein method for the manifold $\S^1 = \{x \in \R^2 : |x|^2=1\}$, and in particular, provide examples of bounds between popular distributions within directional statistics. We shall take a kernel based route similar to that of \cite{ley2017distances} which avoids the need to directly bound the solution to the Stein equation. We shall motivate as to why below.

\subsection{Motivation and Background}

As described in the Introduction,  recent development within the area of Stein's method \cite{le2020diffusion} and \cite{thompson2020approximation} have led to the ability to extend the diffusion method originally presented in \cite{mackey2016multivariate} to general manifolds. This particular construction of Stein's method, however, relies upon a sufficient condition which links the geometry of the manifold with the probability density $p = e^{-\phi}$, namely that
\begin{equation}\label{eq:sufficient condition}
 \Ric+ \Hess^\phi \geq 2\kappa g 
\end{equation}
for some $\kappa>0$ with $\Ric$ the Ricci curvature on the manifold and $\Hess^\phi$ the Hessian of $\phi$.
For $\S^1$, $\Ric = 0$ everywhere and so this condition is simplified to a convexity requirement on $\phi$ (a log-convexity condition on $p$).
\begin{remark}
This is the sufficient condition initially presented in \cite{le2020diffusion}. An identical condition is put forward in \cite{thompson2020approximation} which redefines the Bakry-\'Emery-Ricci tensor.
\end{remark}
However, many popular distributions that are used in directional statistics do not satisfy this sufficient condition for any choice of their canonical parameters; e.g. von-Mises, Bingham, uniform, cardioid and wrapped distributions. Therefore, the diffusion approach will not be utilised for our study. This motivates the need to use classical methods in order to construct a Stein method for distributions on $\S^1$.

One may think of applying already known density methods on an interval $[a,b]$ and identify this as the circle. However, one must appreciate that in order to equate the circle with an interval, it is neccessary to assign a wrapping at the endpoints of the interval. This means one can not simply employ general density methods discussed in, for example, \cite{chen2010normal}. Boundary conditions on $f_h$ and $p$ must be obtained for these methods to be applicable --- particularly, $\lim_{x \rightarrow a+}f(x)p(x) = \lim_{x \rightarrow b-} f(x)p(x)$. The von-Mises, Bingham and uniform distributions do not satisfy this boundary condition unless one restricts the function space for which the operator is defined on. Instead, we will modify the density approach to accommodate the geometry of $\S^1$, and we shall see that by the definition of continuous functions on $\S^1$, this condition is always satisfied for absolutely continuous $f$ and $p$.  Even though one may use an interval of arbitrary length to construct the circle, we shall, in this paper, use intervals of lengths $2\pi$ for simplicity.

The Stein kernel, which shall be defined later, has been shown \cite{ley2017distances,ghaderinezhad2018general} to provide another way to construct analytic bounds on the Wasserstein metric between two known distributions. For this reason, we shall be utilising the Stein kernel to bound the Wasserstein metric. This avoids the need to bound the solution to the Stein equation directly, which can sometimes yield large bounds. For example, when looking at the von-Mises distribution $X \sim VM( 0,\kappa)$, one will find that the solution to the Stein equation has the form (up to a constant proportional to $e^{-\kappa \cos x}$),
\[ f_h(x) = e^{-\kappa \cos(x)} \int_{-\pi}^x (h(u)-\E[h(X)])e^{\kappa \cos(u)} du.\]
This cannot be bounded via conventional means, i.e. using properties of the CDF, due to the oscillatory nature of the cosine function in the exponent. Bounding it directly using the Lipschitz continuity of $h$ will result in very large upper bounds. 
 In addition to this, one usually relies upon applying one or more of: exchangeable pairs, size-biasing, sum of variables and zero-biasing in order to bound the Wasserstein metric. This will not be needed when working with the kernel as our method --- reminiscent of \cite{ley2017distances} --- shall be directly comparing the operators to obtain an upper bound. As we shall see, there are limitations to which distributions we can compare, which is the price to pay for using this kernel based method.

\subsection{Main Results}

In this paper, we present a formulation for a Stein's method on $\S^1$. We develop a new Stein kernel --- named circular Stein kernel --- that is invariant to choice of coordinate system on $\S^1$. Let $X\sim P$ be a random variable on $\S^1$ with density function $p$. The circular Stein operator of $P$ is defined in a similar fashion to its Euclidean counterpart, $\mathcal{T}_p f = (fp)'/p$. We then go on to define the circular Stein kernel of $X$ which is defined differently as
\[ \tau^c(x) := \mathcal{T}_p^{-1}\sin(\mu-\mathrm{Id})= \int_{\mu-\pi}^x \sin(\mu-y)p(y) dy \]
where $\mu$ is the mean angle of $X$ defined as $\mathrm{Arg}(\E[e^{iX}])$. 
The first notable result is an upper bound on the Wasserstein metric: 
\newline
\textbf{Theorem \ref{thm: bayesian comparison}.}
\textit{ 
Let $X$ and $Y$ be random variables on $\S^1$ with Lebesgue densities $p_1,p_2$ respectively and $\mathrm{supp}(X)=\mathrm{supp}(Y)=\S^1$ and define $\pi_0(x)=\frac{p_2(x)}{p_1(x)}$. Furthermore, let $\mu$ be the mean angle of $X$ and $\tau^c$ be the circular Stein kernel of $X$. Assume that $p_1,p_2$ and $\pi_0$ are differentiable everywhere on $\S^1$.  
Then we have the following bounds on the Wasserstein metric between $X$ and $Y$:
\[ |\E[\tau^c(X)\pi_0'(X)]|\leq d_W(Y,X) \leq \E[|\alpha(X)\pi_0'(X)\tau^c(X)|],	\]
where 
\[ \alpha(x) = \dfrac{\int_{\mu-\pi}^x (\E[X]-y)p_1(y)dy}{\int_{\mu-\pi}^x\sin(\mu-y)p_1(y)dy}. \]
}

This is the main method of comparison that we employ for random variables on $\S^1$. The second, a direct consequence of the above Theorem, is a comparison between the von-Mises (VM) and wrapped normal (WN) distributions:
\newline
\textbf{Corollary \ref{thm:wrapped normal}.}
\textit{
Let $X~\sim VM(0,\kappa)$ and $Z \sim WN(0,\sigma^2)$. Then the Wasserstein distance between $X$ and $Z$ is bounded above,
\[ d_W(Z,X) \leq \dfrac{2 \pi^3}{\kappa \sigma^4} + 2 \pi. \]
}
This bound on the Wasserstein metric is discovered after an application of Theorem \ref{thm: bayesian comparison}. The result itself provides a relatively weak comparison between the two distributions, however such a result is still of interest. In contrast to the von-Mises distribution, wrapped normal distribution is typically non-trivial to simulate from, and so comparisons between these two distributions have been a central point of discussion within directional statistics \cite{mardia1975statistics,best1979efficient}. 
 
\textbf{Notation and conventions. }Throughout the paper we shall be using the following notation: $P$ is a probability measure on $\S^1$ with continuous Lebesgue density $p$. $L^1(P)$ denotes the space of absolutely integrable functions on $\S^1$ under $P$. We use this abbreviation for $L^1(\S^1,P)$ unless explicitly stated otherwise. For simplicity, we shall assume that the support of $P$ is a connected subset of $\S^1$. Any reference to standard coordinates of $\S^1$ means that we associate $\S^1$ with the interval $(-\pi,\pi]$ alongside the equivalence relation $-\pi\sim \pi$; $x \sim y$ means that $x-y\; \mathrm{mod}\; 2\pi =0$. We prescribe $\S^1$ with its canonical Riemannian metric $g=dx^2$.

\section{The Circular Stein Operator}\label{sec:op}

This initial section is dedicated to establishing the framework necessary to formulate the Stein equation on $\S^1$. We begin by defining the canonical Stein operator for $\S^1$ and further define its inverse operator. In lieu of a diffusion approach, we shall pursue a modified density approach which draws inspiration from D\"obler \cite{dobler2015stein}.

Before we start, we recall two definitions from analysis that shall aid us: 
\begin{definition}
A function $f$ is absolutely continuous on $[a,b]$ if $f$ has derivative $f'$ almost everywhere, $f' \in L^1((a,b],dx)$ and one can write
\[ f(x) = f(a)+\int_a^xf'(y) dy, \quad a \leq x \leq b. \]
\end{definition}
\begin{definition}
Let $g\in C^\infty((a,b])$ with $g(a)=g(b)=0$ be a given function. We say a function $\phi\in L^1((a,b],dx)$ has weak derivative $\phi'$ if
\[ \int_a^b \phi(x) g'(x) dx = -\int_a^b \phi'(x) g(x) dx.\]
\end{definition}

\subsection{The Canonical Density Operator}

To begin, we start by adapting the general density method for $\S^1$:
\begin{definition}\label{def:stein class}
Let $P$ be a probability measure on $\S^1$ with Lebesgue density $p$ with the assumption that $p' \in L^1(dx)$. Define $\mathcal{I}=\{x \in \S^1 : p(x)> 0\}$. The Stein class $\F(P)$ of $P$ is the collection of functions $f:\S^1\rightarrow \R$ such that 
\begin{enumerate}[i)]
\item $f$ is differentiable everywhere on $\mathcal{I}$,
\item $f' \in L^1(dx)$,
\item $\int_{\mathcal{I}}(fp)' dx=0$.
\end{enumerate}
\end{definition}
Since we assume $p$ is absolutely continuous, it is immediate that for any $f \in \F(P)$ the product $fp$ is absolutely continuous on $\mathcal{I}$ since $f$ is also absolutely continuous by items i) and ii).
Because of the Lebesgue integrability assumption on $p'$, constant functions are always in $\F(P)$, and hence $\F(P)$ is always non-empty.

\begin{definition}\label{def:stein pair}
The Stein operator $\mathcal{T}_p$ of a probability measure $P$ on $\S^1$ is the mapping
\[
\mathcal{T}_p:\F(P)\rightarrow L^1(P)
\]
given by
\begin{equation}\label{eq: stein operator}
\mathcal{T}_p f(x) = \begin{cases} 
\frac{(fp)'}{p}(x)& x \in \mathcal{I}\\ 
f(x) &  x \notin \mathcal{I}.
\end{cases}
\end{equation} 
Furthermore, the double $(\F(P),\mathcal{T}_p)$ is called the Stein pair.
\end{definition}

If one were to compare these definitions with their Euclidean counterparts (for example in \cite{ley2017stein}), one would see immediate differences. Instead of requiring $fp$ to be absolutely continuous, we are able to instead restrict $f' \in L^1(dx)$ so that $fp$ is absolutely continuous. In other words, we need not demand absolute continuity in $fp$ since it is a consequence of Definition \ref{def:stein pair}. Stating that $f'$ is differentiable everywhere allows us to write down, explicitly, the Stein operator as a differential operator in terms of $f$ --- see below. The key difference is that for the definition of the Stein class on $\R$, the third condition $ \int_{\R} (fp)' dx = 0 $ is required (cf. \cite{ley2017stein}) for $p \in L^1(\R,dx)$. In the case of the circle, if $\mathrm{supp}(P)=\S^1$, this condition is automatically satisfied: if we are to identify $\S^1$ with the interval $[-\pi,\pi)$ alongside the equivalence relation $-\pi\sim\pi$, then $f(-\pi)p(-\pi)=f(\pi)p(\pi)$ and so 
\[ \int_{\S^1}(fp)'dx = (fp)\bigg|_{-\pi}^\pi=f(\pi)p(\pi)-f(-\pi)p(-\pi)=0.\]
\begin{lemma}\label{lemma:zero}
Let $P$ be a probability measure on $\S^1$ with Lebesgue density $p$ and Stein class $(\F(P), \mathcal{T}_p)$. For all $f \in \F(P)$, $\E_P[\mathcal{T}_pf] =0$.
\end{lemma}
\begin{proof}
This statement is evident from Definition \ref{def:stein class}. For $f \in \F(P)$ and $\mathcal{I}=\{x \in \S^1 : p(x) >0\}$,  
\[\E_P[\mathcal{T}_p f] = \int_{\mathcal{I}} \dfrac{(fp)'}{p} p  dx + \int_{\mathcal{I}^c} fp dx=0.
\]
\end{proof}
\begin{example}
We now give some examples of Stein operators for circular distributions, with $f \in \F(P)$ throughout;
\begin{enumerate}[a)]
\item Uniform measure $U(\S^1)$ with $p(x)=(2\pi)^{-1}$;
\begin{equation}\label{eq:uniform op}
\mathcal{T}_pf(x) = f'(x).
\end{equation}
In this particular instance, the Stein class is
\[ \F(P) = \{f \in C^0(\S^1): f' \in L^1(dx)\}.  \]
\item von-Mises distribution $VM(\mu,\kappa)$ with $p(x) = (2\pi I_0(\kappa))^{-1}\Exp(\kappa\cos(x-\mu))$, where $I_n(x)$ is the modified Bessel function of the first kind defined as $ I_n(\kappa):=\frac{1}{\pi}\int_0^\pi e^{\kappa \cos \theta} \cos(n\theta) d\theta; $
\[ \mathcal{T}_p f(x) = f'(x) - \kappa\sin(x-\mu) f(x). \]
\item Bingham distribution $Bing(\mu,\kappa)$ with $p(x)=(2\pi I_0(\kappa/2) \Exp(\kappa/2))^{-1} \Exp(\kappa \cos^2(x-\mu))$;
\[ \mathcal{T}_pf(x) = f'(x) - \kappa\sin(2(x-\mu))f(x). \]
\item Cardioid distribution $C(\mu,\rho)$ with $p(x) = (2\pi)^{-1}(1+2\rho \cos(x-\mu))$ and $|\rho|\leq \frac{1}{2}$;
\[ \mathcal{T}_pf(x) = f'(x) - \dfrac{2\rho \sin(x-\mu)}{2\rho \cos(x-\mu) +1}f(x). \]
\end{enumerate}
\end{example}

\begin{definition}\label{def:mean angle}
Let $X$ be a circular random variable. The mean angle $\mu$ is defined as $\mu=\mathrm{Arg}(\E[e^{iX}])$ for $i^2=-1$ and $\mathrm{Arg}$ is the complex argument function.
\end{definition}
This quantity's origin is from the first circular moment $\E[e^{iX}]$ which is decomposed into $\phi_1=\rho e^{i\mu}$ in the standard coordinate system of $\S^1$; $\rho$ is the mean resultant length, and $\mu$ is also known as the mean direction \cite{mardia2009directional}. Before calculating $\mu$, it is paramount to determine what coordinate system one is working with. The first moment, $\phi_1$, is not necessarily always invariant to choice of coordinate system and $\mathrm{Arg}$ is only defined to have support in the standard coordinates of $\S^1$. One will have to convert $\phi_1$ to be in standard coordinates before proceeding to calculate $\mu$. Under the standard coordinates, one particular property $\mu$ has is that $\E[\sin(X-\mu)]=0$. This fact will be important in the following section.

We shall be using this parameter as a foundation from which we shall construct a coordinate system on $\S^1$ for the purpose of integration. This procedure is as follows: For any $x \in \S^1$ which is not $\mu+\pi$, the antipodal point to $\mu$, there is a unique tangent vector $V_x \in T_\mu \S^1 \cong \R$ with $\sqrt{g(V_x,V_x)}<\pi$ such that $\Exp_\mu(V_x)=x$. From which, the map $x \mapsto V_x$ determines a local coordinate system covering $\S^1\setminus \{\mu+\pi\}$. Furthermore, the mapping identifies $\S^1\setminus \{\mu+\pi\}$ with $(-\pi,\pi)\subset \R.$ Under this new coordinate system, $\mu$ is identified with the origin of $\R$ and $x-\mu$ is simply $V_x$. Then, by mapping $\mu+\pi$ to $\pi$, we in effect identify $\S^1$ with $(-\pi,\pi]\subset \R$ with the understanding that the two endpoints are wrapped together; $-\pi$ is identified with $\pi$. In the case where $\mu$ is not unique, for example with the uniform measure on $\S^1$, we take (any) one of the valid values for $\mu$ and form the corresponding identification as described above. Hence, our chosen coordinate system of $\S^1$ is dependent upon $P$. In the sequel, any reference to the \textit{$\mu$ coordinate system} will directly refer to this construction.

\subsection{The Inverse Operator}

The next objective is to define the inverse of the Stein operator \eqref{eq: stein operator} from which we can define the Stein kernel.  Under the \textit{$\mu$ coordinate system}, since we have identified $\mu$ with $0$, $\E_P[\sin(X)]=0$ for a random variable $X$ on $\S^1$. Moreover, $p$ is now centred at $\mu$, meaning that $\mathrm{Arg}(\E[e^{i(X-\mu)}])=0$. For example, if $P$ is the von-Mises, $P \sim VM(\mu,\kappa)$ in standard coordinates, $P$ in the \textit{$\mu$ coordinate system} changes to $P \sim VM(0,\kappa)$.

\begin{definition}\label{def:inverse}
Let $\F^0(P) = \{h \in L^1(P) : \E_P[h]=0\}$ and define the operator $ \mathcal{T}^{-1}_p:\F^0(P)\rightarrow \F(P) $ by
\begin{equation}\label{eq:inverse}
\mathcal{T}^{-1}_p h(x) :=
\begin{cases} \dfrac{1}{p(x)}\int_{-\pi}^x h(y) p(y) dy + \dfrac{h(-\pi)p(-\pi)}{p(x)} & \mathrm{if }\; p(x) \neq 0,\\
h(x) & \mathrm{if }\; p(x) = 0,
\end{cases}
\end{equation}
in which the parameters of $p$ are in terms of the \textit{$\mu$ coordinate system}.
\end{definition}
\begin{proposition}\label{prop:inverse}
The operator $\mathcal{T}_p^{-1}$ is the inverse of $\mathcal{T}_p$.
\end{proposition}
\begin{proof}
There are two sections to this proof: the first being the case where $p(x)\neq 0$ and the second being the case where $p(x)=0$. We begin with the first case.
First, let us check that for a function $h \in \F^0(P)$, $\mathcal{T}_p\mathcal{T}^{-1}_p h = h$:
\begin{align*}
\mathcal{T}(\mathcal{T}^{-1}_p h)(x)&= \dfrac{((\mathcal{T}^{-1}_ph(x)) p(x))'}{p(x)},\\
&= \dfrac{1}{p(x)}\dfrac{\partial}{\partial x} \bigg(\int_{-\pi}^x h(y) p(y) dy + h(-\pi)p(-\pi)\bigg),\\
&=\dfrac{1}{p(x)} h(x)p(x),\\
&=h(x).
\end{align*}
Now to show the other way, let $h \in \F(P)$. Since $\E_P[\mathcal{T}_p h]=0$ by Lemma \ref{lemma:zero}, it is clear that $\mathcal{T}_p h \in \F^0(P)$. Then,
\begin{align*}
\mathcal{T}_p^{-1} (\mathcal{T}_p h)(x) &=\dfrac{1}{p(x)}\int_{-\pi}^x \mathcal{T}_p h(y) p(y)dy + \dfrac{h(-\pi)p(\mu-\pi)}{p(x)},\\
&=\dfrac{1}{p(x)}\int_{-\pi}^x \dfrac{((h(y)p(y))'}{p(y)}p(y) dy+\dfrac{h(-\pi)p(-\pi)}{p(x)},\\
&=\dfrac{1}{p(x)}\int_{-\pi}^x (h(y)p(y))' dy + \dfrac{h(-\pi)p(-\pi)}{p(x)},\\
&=\dfrac{1}{p(x)}\big(h(x)p(x) - h(-\pi)p(-\pi) + h(-\pi)p(-\pi)\big),\\
&=h(x).
\end{align*}
For the case where $p(x)=0$, let $h \in \F^0(P)$. Then
\[
\mathcal{T}_p (\mathcal{T}_p^{-1} h)(x) = \mathcal{T}_p h(x) = h(x).
\]
For the other way, since $p(x)=0$ it is clear that $\E_P[\mathbb{I}_{\mathcal{I}^c} \mathcal{T}_p h]=0$ and hence $\mathcal{T}_p h \in \F^0(P)$ on $\mathcal{I}^c$, therefore
\[ \mathcal{T}_p^{-1}(\mathcal{T}_p h)(x) = \mathcal{T}_p^{-1}h(x)=  h(x). \]
\end{proof}
A special quantity is obtained when we select $h=\nu-\mathrm{Id}$ in (\ref{eq:inverse}), where $\nu=\int_{-\pi}^\pi x p(x) dx$. Applying the inverse operator to this particular $h$ we generate the classical Stein kernel: when $p(x)\neq 0$
\begin{align*}\label{eq: classic kernel}
\tau(x) :&= \mathcal{T}_p^{-1}(\nu-\mathrm{Id})(x)\\  &= \dfrac{1}{p(x)}\int_{-\pi}^x (\nu-y)p(y) dy+ \dfrac{(\nu+\pi) p(-\pi)}{p(x)}\numberthis,\\
&=-\dfrac{1}{p(x)}\int_x^{\pi} (\nu-y)p(y)+ \dfrac{(\nu+\pi) p(\pi)}{p(x)}. 
\end{align*}
Again, we must define $\tau(x)=\nu-x$ when $p(x)=0$. However, this follows from Definition \ref{def:inverse}.

The role of the constant $h(-\pi)p(-\pi)$ is to ensure that $\mathcal{T}_p^{-1}$ is injective. If one does not include it, then we have that $\mathcal{T}_p^{-1} (\mathcal{T}_p h)(-\pi) = 0$ and $\mathcal{T}_p^{-1} (\mathcal{T}_p h)(\pi)=0$ which is not necessarily the case for general $h$. However, owing to the definition of the Stein operator we also have the following:
\begin{corollary}\label{corollary:inverse}
Fix any $C \in \R$. Define $g(x) = \mathcal{T}_p^{-1} h(x) + C/p(x) $ for $x \in \mathcal{I} $  and $g(x)=h(x)$ for $x \notin \mathcal{I}$. Then, 
\[ \mathcal{T}_p g(x) = \mathcal{T}_p (\mathcal{T}_p^{-1} h)(x). \]
\end{corollary}

\begin{example}\label{exmp:uniform kernel}
Let $X\sim U(\S^1),$ the uniform measure on $\S^1$. Then $\nu=\int_{-\pi}^\pi x /2\pi dx = 0$ and choose $\mu=0$. The Stein kernel of this distribution is 
\begin{align*}\label{eq:uniform kernel}
\tau(x) &= 2\pi \int_{-\pi}^x - \dfrac{y}{2\pi}dy+\pi,\\
&= \dfrac{\pi^2 - x^2}{2} + \pi. \numberthis
\end{align*}
\end{example}

One of the main uses of the Stein kernel is to be able to construct bounds on the Wasserstein distance between distributions on $\R$ (cf. Theorem 3.1 in \cite{ley2017distances}). If we wish to adapt this theorem onto $\S^1$ for a circular distribution --- say a von-Mises distribution --- we will have to compute the Stein kernel. So far we have only looked at examples with the uniform measure on $\S^1$ due to its simple Lebesgue density.
However, for the von-Mises distribution in particular, one will quickly find that obtaining a closed form solution of the kernel is impossible. 
For simplicity, let $X\sim VM(0,\kappa)$; then
\[ \tau(x) = \mathcal{T}_p(\E[X]-\mathrm{Id})(x) = e^{-\kappa\cos(x)}\int_{-\pi}^x (\mu-y)e^{\kappa \cos(y)} dy + \dfrac{C}{p(x)}. \]
There are two problems with this: The first is that the integral is intractable. We can obtain bounds on $\tau(x)$, but these bounds do not particularly aid in bounding the Wasserstein metric as they will be  large --- this is akin to bounding the solution to the Stein equation which is what we wanted to avoid. The second is the definition of $\mu$. In Example \ref{exmp:uniform kernel} we used $\mu=\int_{-\pi}^{\pi}xp(x) dx$, but for directional data analysis this is not used as a parameter of location since the standard mean is not well defined on $\S^1$.

A different approach is to redefine $\mu$ to be the extrinsic mean of $\S^1$: $\E[e^{iX}]$. This does, however, require us to completely redefine the kernel to ensure that $\tau(\pi)=0$. This is precisely the route that we shall take in the next section except we shall not use the full extrinsic mean, rather the mean angle as defined in Definition \ref{def:mean angle}. We shall see that this particular choice of parameter coincides well with the von-Mises and other common directional distributions.

\section{The Circular Stein Kernel}\label{sec:kernel}

The end of the previous section lead us to motivate the need to redefine the Stein kernel for certain circular distributions. We dedicate this section to constructing this new kernel as well as compute it for a handful of distributions. Similarly to the classical Stein kernel, we shall utilise the inverse Stein operator to initially define it.

\begin{definition}
Let $X$ be a circular random variable with distribution $P$, mean angle $\mu = \mathrm{Arg}(\E[e^{iX}])$ and $\mu$-centred Lebesgue density $p$; then $\sin(x) \in \F^0(P)$ on the \textit{$\mu$ coordinate system}. The circular Stein kernel $\tau^c$ of $P$ is defined as
\begin{align*}
 \tau^c(x) &:=\mathcal{T}^{-1}\sin(-\mathrm{Id})(x),\\
 &=-\dfrac{1}{p(x)}\int_{-\pi}^x \sin(y)p(y) dy,\\
 &=\dfrac{1}{p(x)}\int_x^{\pi} \sin(y)p(y) dy.
\end{align*}
\end{definition}
By a $\mu$-centred density of a circular random variable $X\sim P$ with density $p(x;\mu)$, we mean $p(x;0)$. In Example \ref{exmp:vm kernel}, we shall see that this $\mu$-centred density is precisely the density of the random variable $X-\mu\; \mathrm{mod} 2\pi$.

We are distinguishing the circular Stein kernel from the classical Stein kernel (\ref{eq: classic kernel}) with superscript $c$.

\begin{example}\label{exmp:vm kernel}
Let $X\sim VM(\mu,\kappa)$ with Lebesgue density $p(x)=(2\pi I_0(\kappa))^{-1} \Exp(\kappa\cos(x-\mu))$ with $\mu \in \S^1$ and $\kappa>0$.
To calculate the mean angle, we first introduce the special function $I_1(\kappa) = \frac{1}{2\pi}\int_0^{2\pi} \cos(x) e^{\kappa \cos(x)} dx$. It turns out that dividing the moment by $e^{i\mu}$ aids in its calculation:
\begin{align*}
\E[e^{i(X-\mu)}] &= \frac{1}{2\pi I_0(\kappa)} \int_{\S^1} e^{i(x-\mu)}e^{\kappa\cos(x-\mu)}dx,\\
&=\dfrac{1}{I_0(\kappa)}\bigg(\dfrac{1}{2\pi}\int_{\S^1} \cos(x-\mu) e^{\kappa \cos(x-\mu)} dx,\\
 &\quad\quad\quad+ \dfrac{i}{2\pi} \int_{\S^1} \sin(x-\mu) e^{\kappa \cos(x-\mu)} dx\bigg),\\
&=\frac{I_1(\kappa)}{I_0(\kappa)},
\end{align*}
since $\sin(x)e^{\kappa\cos(x)}$ is anti-symmetric about the origin. Therefore, because $\mathrm{Arg}(\E[e^{i(X-\mu)}])=0$, it must be that $\mu=\mathrm{Arg}(\E[e^{iX}])$.
 Then we calculate the circular Stein kernel by switching to \textit{$\mu$ coordinates},
\begin{align*}
\tau^c(x) &= \Exp(-\kappa \cos(x))\int_{-\pi}^x -\sin(y) \Exp(\kappa\cos(y)) dy,\\
&= \dfrac{1}{\kappa}-\dfrac{1}{\kappa}\Exp\big(\kappa(-1-\cos(x))\big).
\end{align*}
Notably, we have the bounds
\begin{equation}\label{ineq: kappa}
0 \leq \tau^c(x) \leq \dfrac{1}{\kappa}(1-e^{-2\kappa})\leq \dfrac{1}{\kappa}
\end{equation}
which achieves minima at $x=\pm \pi$ and a maximum at $x=0$.  This particular bound on $\tau^c$ for the von-Mises distribution will be of use to us later on.
\end{example}

\begin{example}\label{exmp: bing}
Let $X$ be a one-dimensional Bingham random variable, with Lebesgue density 
\[ p(x) = \dfrac{1}{2\pi e^{\kappa/2}I_0(\frac{\kappa}{2})} \Exp\big(\kappa \cos^2(x-\mu)\big),\quad x \in \S^1. \]
One can deduce that the mean angle is $\mu$ due to the fact that $p$ is symmetric about $\mu$, and so in standard coordinates $\E[\sin(X-\mu)]=0$.
In order to calculate the circular Stein kernel of this random variable, we must first compute the integral
\begin{align*}
-\int_{-\pi}^{x} \sin (y) \Exp \big(\kappa \cos^2 (y)\big)du &= \int_{-\sqrt{\kappa}}^{\sqrt{\kappa}\cos (x)}\dfrac{e^{z^2}}{\sqrt{\kappa}} dz,\\
&=\dfrac{\sqrt{\pi}}{2\sqrt{\kappa}}\big(\mathrm{erfi}(\sqrt{\kappa}\cos (x)) + \mathrm{erfi}(\sqrt{\kappa})\big).
\end{align*}
Here, $\mathrm{erfi}(x)$ is the imaginary error function and relates to the error function $\mathrm{erf}(x) = i \mathrm{erfi}(ix)$.
Whence, 
\[ \tau^c(x) = \dfrac{\sqrt{\pi}}{2}\dfrac{e^{-\kappa \cos^2 (x)}}{\sqrt{\kappa}}\big(\mathrm{erfi}(\sqrt{\kappa}\cos (x)) + \mathrm{erfi}(\sqrt{\kappa})\big). \]
\end{example}

\begin{example}
Let $X \sim U(\S^1)$ the uniform measure on $\S^1$ which has Lebesgue density $p(x) = \frac{1}{2\pi}$, $x \in \S^1$, and choose: $\mu=0$ from Example \ref{exmp:uniform kernel}.
\[ \tau^c(x) = 2\pi \int_{-\pi}^x -\dfrac{\sin(y)}{2\pi} dy = \cos(x)+1. \]
One can also obtain this kernel by taking the limit as $\kappa \rightarrow 0$ in Example \ref{exmp:vm kernel} for $\mu=0$;
\begin{align*}
\lim_{\kappa \rightarrow 0} \dfrac{1-e^{\kappa(-1-\cos x)}}{\kappa}&= \lim_{\kappa \rightarrow 0} (1+\cos x) e^{\kappa(-1-\cos x)} ,\\
&=1+\cos x.
\end{align*}
\end{example}

Similar to the classical Stein kernel, the circular Stein kernel also satisfies the following integration by parts property.
\begin{lemma}\label{lem:int by parts kernel}
Define $X$ to be a random variable on $\S^1$ with corresponding circular Stein kernel $\tau^c$ and mean angle $\mu$, and let $\phi$ be absolutely continuous with weak derivative $\phi$. We have that
\[ \E[\sin(X)\phi(X)]=\E[\tau^c(X)\phi'(X)].\]
\end{lemma}
\begin{proof}
Let $X$ have Lebesgue density $p$ on $\S^1$, then
\[
\E[\tau^c(X)\phi'(X)]=-\int_{\S^1}\int_{-\pi}^x \sin(y)p(y)dy \phi'(x)dx.
\]
Using integration by parts with $u=\int_{-\pi}^x \sin(y)p(y)dy$ and $v'=\phi'(x)$ we obtain
\begin{align*}
\E[\tau^c(X)\phi'(X)] &= -\phi(x)\int_{-\pi}^x \sin(y)p(y) dy\bigg|_{-\pi}^{\pi}+\int_{\S^1}\sin(x)\phi(x)p(x) dx\\
&=\E[\sin(X)\phi(X)].
\end{align*}
In the second equality we have used the continuity of $\phi$ and the fact that $\E_P[\sin(X)]=0$ in the \textit{$\mu$ coordinate system}.
\end{proof}

\section{Bounding of the Wasserstein Distance}\label{sec:wass}

Let ${W}=\{h : \norm{h} \leq 1 \}$ be the set of Lipschitz continuous functions with a Lipschitz constant of 1. The Wasserstein distance between two probability measures $P_1$ and $P_2$ on measurable space $(\Omega,\F)$ is defined as 
\[ d_{W}(P_1,P_2) = \sup_{h \in W}\bigg|\int_\Omega h dP_1 - \int_\Omega h dP_2\bigg|. \]
Using the Stein operator, we may also construct the renowned Stein equation for $X\sim p$;
\begin{equation}\label{eq:stein eq s1}
\mathcal{T}_p f_h(x) = h(x)-\E[h(X)]
\end{equation}
with $\mathcal{T}_p$ defined in Definition \ref{def:stein pair}.
Clearly, $\E[\mathcal{T}_pf_h(X)]=0$ since $\E[h(X)-h(X)]=0.$ More concretely, we can say that $h-\E[h(X)]\in \F^0(P)$. It is now evident that we can apply the inverse Stein operator to both sides of the Stein equation (\ref{eq:stein eq s1}) in order to find its solution. However, by Corollary \ref{corollary:inverse} we may choose $C = -p(-\pi)(h(-\pi)-\E[h(X)])$ so that we can define the solution
\begin{equation}\label{eq:solution}
f_h(x):= \dfrac{1}{p(x)}\int_{-\pi}^x (h(y)-\E[h(X)])p(y) dy.
\end{equation}

\subsection{Main Theorem}

We next turn our attention to the use of the Stein kernel to bound the Wasserstein distance for distributions on $\S^1$. We shall take a similar approach to that of \cite{ley2017distances} with modifications of the kernel that are discussed in \cite{dobler2015stein} since this does not involve bounding the solution to the Stein equation directly. 
\begin{lemma}\label{lemma: 3.13}
Let $\tau^c$ be the circular Stein kernel of a circular random variable $X$ with Lebesgue density $p$ and mean angle $\mu$. Define the solution to the Stein equation $f_h$ by (\ref{eq:solution}) and further define $g_h=f_h/\tau^c$.
Then, we have for any Lipschitz continuous test function $h:\S^1 \rightarrow \R$
\[  |g_h(x)| \leq \norm{h'} \dfrac{\int_{\mu-\pi}^x (\E[X] - y)p(y) dy}{ \Big|\int_{\mu-\pi}^x \sin(\mu-y) p(y) dy \Big|}.  \]
\end{lemma}
\begin{remark}
This result was formulated by D\"obler in \cite{dobler2015stein} Proposition 3.13 a). Particularly in D\"obler's work, he looked at a general kernel on an interval of $\R$ with closure $\overline{(a,b)}$. This kernel took the form
\[ \eta(x) = \dfrac{1}{p(x)}\int_a^x \gamma(t)p(t) dt. \]
In the proposition, D\"obler imposed conditions on the $\gamma$ that can be used; in particular, $\gamma$ is decreasing on $\overline{(a,b)}$. However, this condition is not necessary for part a) of the relevant proposition and instead relies upon properties of $h$ and the CDF of $p$. Therefore this lemma is easily translated from an interval onto the circle.
\end{remark}

\begin{theorem}\label{thm: bayesian comparison}
Let $X$ and $Y$ be circular random variables with Lebesgue densities $p_1,p_2$ respectively and $\mathrm{supp}(X)=\mathrm{supp}(Y)=\S^1$, define $\pi_0(x)=\frac{p_2(x)}{p_1(x)}$. Furthermore, let $\mu$ be the mean angle of $X$ and $\tau^c$ be the circular Stein kernel of $X$. Assume that $p_1,p_2$ and $\pi_0$ are differentiable everywhere on $\S^1$.  
Then we have the following bounds on the Wasserstein metric between $X$ and $Y$:
\[ |\E[\tau^c(X)\pi_0'(X)]|\leq d_W(Y,X) \leq \E[|\alpha(X)\pi_0'(X)\tau^c(X)|],	 \]
where 
\[ \alpha(x) = \dfrac{\int_{\mu-\pi}^x (\E[X]-y)p_1(y)dy}{\int_{\mu-\pi}^x\sin(\mu-y)p_1(y)dy}. \]
\end{theorem}
Note that $\alpha$ is expressed in $\mu$\textit{-coordinates}.
\begin{proof}
We begin by proving the lower bound.

First, note that since sine is a Lipschitz continuous function with a Lipschitz constant of 1,
\[ |\E[\sin(Y)]-\E[\sin(X)]|\leq d_W(Y,X).\]
Moreover, since $\mu$ is the mean angle of $X$, the second expectation on the left hand side is 0. For the first expectation,
\begin{align*}
\E[\sin(Y)]&= \int_{\S^1}\sin(x)p_2(x) dx,\\
&=\int_{\S^1}\sin(x) \dfrac{p_2(x)}{p_1(x)}p_1(x)dx,\\
&=\E[\sin(X)\pi_0(X)].
\end{align*}
Then by applying Lemma \ref{lem:int by parts kernel} with $\phi=\pi_0$, we obtain the lower bound. 

For the upper bound, let $(\F_1,\mathcal{T}_1)$ and $(\F_2,\mathcal{T}_2)$ be the Stein pairs of $X$ and $Y$ respectively. Then by Definition of the Stein equation, one clearly sees that $f_h:=\mathcal{T}^{-1}_1(h-\E[h(X)])\in \F_1 $ since $h-\E[h(X)]\in \F^0_1$. We need to verify that $f_h\in \F_2$:

First $f'\in L^1(dx)$ and $f_h$ is differentiable everywhere on $\S^1$ already, because $f_h\in \F_1(p)$. Furthermore $\int_{\S^1} (fp_2)' dx = fp_2|_{-\pi}^\pi=0$ by continuity.
Whence, we can conclude that $f_h \in \F_2$, and more importantly $f_h \in \F_1 \cap \F_2$.
Using this fact, we wish to relate the Stein operators of $X$ and $Y$; $ \mathcal{T}_2(f_h)(x) = \dfrac{\partial f_h(x)}{\partial x} + \dfrac{p_2'(x)}{p_2(x)}f_h(x) $ and
$\mathcal{T}_1(f_h)(x) = \dfrac{\partial f_h(x)}{\partial x} + \dfrac{p_1'(x)}{p_1(x)}f_h(x). $ One can clearly see that both operators share a common term of $f_h'(x)$, and so
\begin{equation}\label{eq: operators}
\mathcal{T}_2(f_h) - \mathcal{T}_2(f_h) = (\log \pi_0)'f_h.
\end{equation}
Now, by definition of the Stein equation \eqref{eq:solution},
\begin{align*}\label{eq:gh}
\E[h(Y)]- \E[h(X)]&= \E[\mathcal{T}_1(f_h)(Y)],\\
&= \E[\mathcal{T}_1(f_h)(Y)]- \E[\mathcal{T}_2(f_h)(Y)],\\
&=-\E[f_h(Y)(\log\pi_0)'(Y))],\\
&=-\E\bigg[\tau^c(Y)\dfrac{f_h(Y)}{\tau^c(Y)}(\log\pi_0)'(Y)\bigg]\numberthis.
\end{align*}
The second equality is due to the fact that $\E[\mathcal{T}_2(f_h)(Y)]=0$ --- since $f_h \in \F_1 \cap \F_2$, and in the third equality we have used Equation (\ref{eq: operators}). Define the quantity $g_h = f_h/\tau^c$.

Now, using Lemma \ref{lemma: 3.13},
\begin{equation}\label{eq:gh bound}
|g_h(x)|\leq \norm{h'}|\alpha(x)|.
\end{equation}
Compiling \eqref{eq:gh} and \eqref{eq:gh bound} together we obtain the upper bound,
\begin{align*}
d_W(Y,X) &\leq \sup_{h : \norm{h'}\leq 1} \norm{h'} \E[|\tau^c(Y) \alpha(Y)(\log\pi_0)'(Y)|],\\
&=\E[|\tau(Y) \alpha(Y)(\log\pi_0)'(Y)|].
\end{align*}
\end{proof}
\begin{remark}
For probability measures on $\R$, $p_1$ and $p_2$ must obey the following requirements in Theorem 3.1 of \cite{ley2017distances}: 
\[
\bigg(\pi_0(x) \int_{-\pi}^x (h(y)-\E[h(X_1)])p_1(y) dy\bigg)' \in L^1(P_2) 
\]
and 
\[
\lim_{x \rightarrow \pi} \pi_0(x) \int_{-\pi}^x (h(y)-\E[h(X_1)])p_1(y) dy=0.
\]
Both the first and second assumptions are not essential when looking at $\S^1$. This is due to the compactness of the circle and the continuity of functions at $-\pi$ and $\pi$.
\end{remark}

It is worth mentioning that $\alpha(x)$ is not a bounded function --- it has singularities at $x=\pm\pi$. To tackle this problem, we will be multiplying it by an auxiliary function that comes about as a result of $(\log p)'$. For example, in both von-Mises and Bingham cases, $(\log p)'$ will contain a sine function to assist in removing the singularity.

\begin{lemma}\label{lemma:alpha}
Let $X$ be a random variable on $\S^1$ with Lebesgue density $p$ such that $p(x)\neq 0$ on $\S^1$. Suppose, without loss of generality, that $\E[X]=0$ in the Euclidean sense after making use of an appropriate chart. Then the function $|\sin(x) \alpha(x)|\leq 2\pi$ and attains this maximum at $x=\pm 2\pi$.
\end{lemma}
\begin{proof}
We begin by bounding the function from above.
\begin{align*}
\lim_{x \rightarrow -\pi} \alpha(x) \sin(x)&= \lim_{x \rightarrow -\pi}\dfrac{\sin(x)\int_{-\pi}^x -y p(y)dy}{\int_{-\pi}^x-\sin(y)p(y)dy},\\
 &=\lim_{x \rightarrow -\pi} \dfrac{\cos(x)\int_{-\pi}^x -y p(y)dy -x\sin(x) p(x) }{-\sin(x) p(x)}.
\end{align*}
Then when looking at the absolute value of the function,
\begin{align*}
\lim_{x \rightarrow -\pi} |\alpha(x) \sin(x)|&=\lim_{x \rightarrow -\pi}\dfrac{|\cos(x)\int_{-\pi}^x -y p(y)dy -x\sin(x) p(x)| }{|\sin(x) p(x)|},\\
&\leq \lim_{x \rightarrow -\pi} \dfrac{|\cos(x)\int_{-\pi}^x -y p(y)dy|}{|\sin(x) p(x)|} + \pi,\\
&=\lim_{x \rightarrow -\pi} \dfrac{-\cos(x)}{p(x)}\lim_{x \rightarrow -\pi} \dfrac{|\int_{-\pi}^x -y p(y)dy|}{|\sin(x)|} +\pi,\\
\overset{\text{L'h\^{o}p}}&{=} \lim_{x \rightarrow -\pi} \dfrac{xp(x)}{\cos(x)} \lim_{x \rightarrow -\pi} \dfrac{1}{p(x)} + \pi,\\
&=2\pi.
\end{align*}
To show that this is indeed a maximum, first note that the function $\alpha(x)\sin(x)$ is 0 at 0. Denote $m(x)=\int_{-\pi}^x -y p(y)dy$ and $s(x)=\int_{-\pi}^x-\sin(y)p(y)dy$, then
\[ (\alpha(x) \sin(x))' = \frac{1}{s(x)}\bigg(x \sin(x)p(x) - \frac{m(x)}{s(x)} \sin(x) p(x) + \frac{m(x)}{s(x)} \cos(x)\bigg) \]
which is 0 if and only if 
\[ s(x) x \sin(x) p(x) - m(x) \sin(x) p(x) + \cos(x) m(x)=0. \]
This only occurs at the point $x=\pm \pi$, and because $\alpha(-\pi)\sin(-\pi)>0$, it attains a maximum. 
\end{proof}

Note that the assumption of $\E[X]=0$ is satisfied if we transformed to $\mu$-coordinates, since the function $f(x)=x$ is anti-symmetric about the origin.

\subsection{Applications of Theorem \ref{thm: bayesian comparison}}

We begin with a simple example that illustrates the use of Theorem \ref{thm: bayesian comparison}.
\begin{example}\label{exmp: vm bing}
Suppose $X \sim VM(0,\kappa)$ and $Y \sim Bing(0,\zeta)$, the Bingham distribution discussed in Example \ref{exmp: bing}. The ratio of the two densities is $\pi_0 =\Exp(\zeta \cos^2 x - \kappa \cos x)$  and it is easily computed $(\log\pi_0)'= -\sin x(2\zeta \cos x - \kappa)$.

Then, by Theorem \ref{thm: bayesian comparison}, we have that
\[ d_W(Y,X) \leq \dfrac{1}{\kappa} (2\zeta + \kappa) \sup_{x \in \S^1} | \alpha(x) \sin(x)|, \]
where we have used the upper bound from inequality (\ref{ineq: kappa}) to bound the kernel. Then applying Lemma \ref{lemma:alpha} gives a final bound of
\[ d_W(Y,X) \leq \dfrac{4 \pi \zeta}{\kappa} + 2 \pi. \]
\end{example}

In this next example, we shall use Theorem \ref{thm: bayesian comparison} to compare two Bayesian posterior densities with the same likelihood, but different priors. This is an analogous example to Section 4.2 of \cite{ley2017distances} which discusses the influence of priors on a normal model of data for a von-Mises model and von-Mises prior.

Inference of this type was first performed in \cite{mardia1976bayesian} which focussed on the von-Mises Fisher distribution and classes of priors that could be applied to give analytic results. More recently, a more relevant inference has been performed in \cite{damien1999full} which specifically uses a von-Mises model (on $\S^1$) with von-Mises prior.
This type of inference has been previously applied to finding the location of airplane locator transmitters \cite{guttorp1988finding}.
 This example will be have a base von-Mises model, and we wish to compare a uniform prior with a von-Mises prior. 
\begin{example}\label{exmp:bayes}
Let $X_1,...,X_n$ be iid from a $VM(\mu,\kappa)$ distribution. For the purposes of this example, we will be interested in making inference about the mean angle $\mu$. 
The likelihood of $\mu$ is calculated to be
\begin{align*}\label{eq:likelihood}
L(\mu;x)&\propto \prod_{i=1}^n\Exp(\kappa\cos(x_i-\mu)),\\
&=\Exp\big(\kappa R\cos(\mu-\psi)\big)\numberthis,
\end{align*}
with \[R^2=n^2(\bar{C}^2+\bar{S^2}),\; \tan\psi=\dfrac{\bar{S}}{\bar{C}},\] where	 \[\bar{C}=\dfrac{1}{n}\sum_{i=1}^n\cos(x_i),\; \bar{S}=\dfrac{1}{n}\sum_{i=1}^n\sin(x_i).\]
Note that the form of the likelihood (\ref{eq:likelihood}) is that of a von-Mises with location and scale parameters $\psi$ and $\kappa R$ respectively.

With the likelihood set up, we will select two priors on $\mu$. The first prior will be the uniform prior on $\S^1$; $\pi_1(\mu) = \frac{1}{2\pi}$. The other shall be an independent von-Mises random variable; $\mu \sim VM(0,\kappa^*)$. We shall name these models model 1 and model 2 respectively.

For model 1, the posterior density is unchanged, $\mu|X_1,...,X_n \sim VM(\psi,\kappa R)$.
For model 2, the posterior density will change. Similarly to calculations performed above, one can check that 
\begin{align*}
\pi_2(\mu|x) & \propto \Exp\big(\kappa R \cos(\mu-\psi)\big) \Exp\big(\kappa^* \cos\mu\big),\\
&=\Exp\big(R^*\cos(\mu-\psi^*)\big),
\end{align*}
with
\[ {R^*}^2=\kappa^2R^2+(\kappa^*)^2+2\kappa\kappa^*R\cos\psi,\; \tan\psi^*=\dfrac{\kappa R\sin\psi}{\kappa R \cos\psi + \kappa^*};\]
In other words, $\mu|X_1,...,X_n \sim VM(\psi',R')$ under model 2.

Now we can apply Theorem \ref{thm: bayesian comparison} with $X\sim VM(\psi,\kappa R)$ and $Y \sim VM(\psi^*,R^*)$. The quotient between the densities of these two random variables exists and is differentiable everywhere on $\S^1$; $\pi_0\propto \Exp(\kappa^*\cos\mu).$ Using inequality (\ref{ineq: kappa}) we know that, in our case, $\tau^c(x)\leq 1/(\kappa R)$.
Then using the fact that $\sup_{x \in \S^1}|\alpha(x)\sin(x)|=2\pi$, the bound on the Wasserstein metric is
\[ d_W(Y,X)\leq \dfrac{2 \kappa^* \pi}{\kappa n \sqrt{\bar{C}^2+\bar{S}^2}}. \]
This shows a convergence rate of $O(n^{-1})$ for the Wasserstein metric between model 1 and model 2 since $\bar{C}^2 + \bar{S}^2$ is $O(1)$. In other words, the effect of the prior becomes negligible as $n$ increases to infinity.
\end{example} 
As previously mentioned, this example is reminiscent to that of Section 4.2 of \cite{ley2017distances} where Ley et. al. discuss comparing Bayesian posteriors for normal models. Similarly to above, they compare a posterior with uninformative uniform prior against a posterior with normal conjugate prior. Their results show a convergence of $O(n^{-3/2})$ compared to ours of $O(n^{-1})$. This discrepancy in orders is due to the fact that we are not calculating the expectation $\E[|\alpha(X)\pi_0'(X)\tau^c(X)|]$ and are instead bounding it above by some enveloping function; thus, giving us a less optimal order bound.
\begin{example}\label{exmp:wn}
In this example, we shall be comparing the von-Mises distribution with the wrapped normal distribution. In circular statistics, the von-Mises distribution with scale parameter $\kappa$ has been used to approximate the wrapped normal distribution by $\kappa$. One question of interest is how big must $\kappa$ be in order for the von-Mises distribution to be a good approximation for wrapped normal distribution?
Let $Z$ be a wrapped normal distribution with 0 mean and variance $\sigma^2$, denoted $Z \sim WN(0,\sigma^2)$ formulated in the following way \cite{mardia2009directional}:
Suppose $U$ is a distribution on $\R$, then the wrapping of $U$ denoted $V$ is $V=U\; \mathrm{mod}\; 2\pi$. If $U$ has Lebesgue density $p_U$, the Lebesgue density of $V$ is
\[ p_V(\theta)=\sum_{k=-\infty}^\infty p_U(\theta +2\pi k),\quad \theta \in [-\pi,\pi). \]
For our specific example, $U$ is a normal distribution and so the wrapped density is
\[ p_{WN}(\theta) = \dfrac{1}{\sigma \sqrt{2\pi}}\sum_{k=-\infty}^\infty \Exp\bigg(\dfrac{-(\theta+2\pi k)^2}{2 \sigma^2}\bigg).\]
An alternative, and more useful, form of this pdf is to instead use the Jacobi triple product:
\begin{align*}
p_{WN}(\theta) &= \dfrac{1}{2\pi}\prod_{n=1}^\infty (1-e^{-\sigma^2n})(1+e^{-\sigma^2(n-1/2)}e^{i\theta})(1+e^{-\sigma^2(n-1/2)}e^{-i\theta}),\\
&=\dfrac{1}{2\pi}\prod_{n=1}^\infty (1-e^{-\sigma^2n})(1+e^{-2\sigma^2(n-1/2)}+2\cos\theta e^{-\sigma^2(n-1/2)}).
\end{align*}
The quantity we need to calculate in order to use Theorem \ref{thm: bayesian comparison} is $(\log p_{WN})'.$
To begin,
\[\log p_{WN}(\theta)=-\log2\pi + \sum_{n=1}^\infty \log(1-e^{-\sigma^2n}) + \log(1+e^{-2\sigma^2(n-1/2)}+2\cos\theta e^{-\sigma^2(n-1/2)}).\]
The derivative of this is
\begin{align*}\label{eq:wn log der}
(\log p_{WN})'(\theta) &= \sum_{n=1}^\infty \dfrac{-2\sin\theta e^{-\sigma^2(n-1/2)}}{1+e^{-2\sigma^2(n-1/2)}+2\cos\theta e^{-\sigma^2(n-1/2)}},\\
&=\sum_{n=1}^\infty\dfrac{- \sin\theta}{ \cosh(\sigma^2(n-\frac{1}{2})) + \cos\theta }.\numberthis
\end{align*}
By Theorem \ref{thm: bayesian comparison}, an upper bound on the Wasserstein metric between $X\sim VM(0,\kappa)$ and $Z \sim WN(0,\sigma^2)$ is
\begin{align*}
d_W(Z,X) &\leq \dfrac{1}{\kappa}\sup_{\theta \in \S^1} |\alpha(\theta)\big( (\log p_{WN}(\theta))' - (\log p_{VM}(\theta))'\big)|,\\
&\leq \dfrac{1}{\kappa}\bigg(\sup_{\theta \in \S^1} |\alpha(\theta) (\log p_{WN}(\theta)'|  + \sup_{\theta \in \S^1} | \alpha(\theta) (\log p_{VM}(\theta))'|\bigg)
\end{align*}
\[  \]
The second part has already been computed as $ 2 \pi\kappa$ from the past two previous examples. For the first term, we may split it up into two parts:
\[\sup_{\theta \in \S^1} |\alpha(\theta)(\log p_{WN}(\theta))'| =\sup_{\theta \in \S^1} |\alpha(\theta)\sin(\theta)|\sup_{\theta \in \S^1} \sum_{n=1}^\infty \dfrac{1}{\cosh(\sigma^2(n-\frac{1}{2}))+\cos(\theta)}, \]
from the continuity of both functions. Note, for all $x \in \R$, $\cosh x \geq 1$ and $|\cos x| \leq 1$ then since $\sigma^2>0$ and $n \geq 1$, $\cosh(\sigma^2(n-1/2)) >1$ whence the denominator is strictly positive. In order to maximize the value of the sum, we must minimize the denominator in $\theta$. This occurs when $\cos\theta = -1$.
The sum in this case can be bounded above; using Taylor series about 0,
\begin{align*}\label{eq:wn bound}
\sum_{n=1}^\infty \dfrac{1}{\cosh(\frac{\sigma^2}{2}(2n-1))-1}&=\sum_{n\in \mathbb{N}\setminus 2 \mathbb{N}} \dfrac{1}{\cosh(\frac{\sigma^2}{2}n)-1},\\
&\leq \sum_{n\in \mathbb{N}\setminus 2 \mathbb{N}} \dfrac{1}{\frac{1}{2}(\frac{\sigma^2}{2}n)^2},\\
&=\dfrac{8}{\sigma^4}\sum_{n\in \mathbb{N}\setminus 2 \mathbb{N}} \dfrac{1}{n^2},\\
&=\dfrac{\pi^2}{\sigma^4}.\numberthis
\end{align*}
Combining this with the previous supremum yields
\begin{align*}
d_W(Z,X) &\leq  \dfrac{1}{\kappa}\bigg(\sup_{\theta \in \S^1} |\alpha(\theta)\sin(\theta)|\sup_{\theta \in \S^1} \sum_{n=1}^\infty \dfrac{1}{\cosh(\sigma^2(n-\frac{1}{2}))+\cos(\theta)}  + \sup_{\theta \in \S^1} | \alpha(\theta) \kappa \sin\theta|\bigg),\\
&\leq \dfrac{1}{\kappa}\bigg(\dfrac{2 \pi^3}{\sigma^4} + 2 \pi \kappa\bigg)=\dfrac{2 \pi^3}{\kappa \sigma^4}+2\pi.
\end{align*}
When $\kappa=\frac{1}{\sigma^2}$ and we let $\kappa \rightarrow \infty$, the Wasserstein metric is bounded by
\[ d_W(Z,X) \leq 2\pi. \]

\end{example}

\begin{corollary} \label{thm:wrapped normal}
Let $X~\sim VM(0,\kappa)$ and $Z \sim WN(0,\sigma^2)$. Then the Wasserstein distance between $X$ and $Z$ is bounded above,
\[ d_W(Z,X) \leq \dfrac{2 \pi^3}{\kappa \sigma^4} + 2 \pi. \]
\end{corollary}

\begin{example}\label{exmp:wc}
Let $Z \sim WN(0,\sigma^2)$ and $X \sim WC(0,\gamma)$, the wrapped Cauchy distribution which is defined similarly to the wrapped normal distribution in Example \ref{exmp:wn}. The wrapped Cauchy distribution has probability density 
\[ p_{WC}(\theta) = \dfrac{1}{2\pi}\dfrac{\sinh\gamma}{\cosh\gamma-\cos\theta},\quad \theta \in \S^1,\gamma>0 \]
which can be obtained by applying the geometric series formula on to the pdf when written in terms of the characteristic function.

We shall be using the wrapped Cauchy as our basis for the comparison; it has circular Stein kernel equal to 
\[ \tau^c(\theta) = (\cosh\gamma-\cos\theta)\log\bigg(\dfrac{\cosh\gamma-\cos\theta}{\cosh\gamma+1}\bigg) \]
and its log derivative is
\[ (\log p_{WC})'(\theta) = -\dfrac{\sin\theta}{\cosh\gamma-\cos\theta}. \]
Since this derivative contains a sine function, we may apply Theorem \ref{thm: bayesian comparison} and Lemma \ref{lemma:alpha} together with the log derivative of the wrapped normal pdf \eqref{eq:wn log der} to obtain a bound on the Wasserstein metric
\begin{align*}
d_W(Z,X) \leq & 2\pi\sup_{\theta \in \S^1} \bigg|(\cosh\gamma-\cos\theta)\log\bigg(\dfrac{\cosh\gamma-\cos\theta}{\cosh\gamma+1}\bigg)\bigg|\\
&\times \sup_{ \theta \in \S^1} \bigg|\dfrac{1}{\cosh\gamma-\cos\theta}+ \sum_{n=1}^\infty\dfrac{1}{ \cosh(\sigma^2(n-\frac{1}{2})) + \cos\theta }\bigg|.
\end{align*}
We next apply the bound obtained in \eqref{eq:wn bound} to bound this from above;
\begin{align*}
d_W(Z,X) \leq & 2\pi\sup_{\theta \in \S^1} \bigg|(\cosh\gamma-\cos\theta)\log\bigg(\dfrac{\cosh\gamma-\cos\theta}{\cosh\gamma+1}\bigg)\bigg|\\
&\times \bigg(\dfrac{1}{\cosh\gamma-1} + \dfrac{\pi^2}{\sigma^4}\bigg).
\end{align*}
Maximizing the first term then yields an upper bound on the Wasserstein metric,
\[ d_W(Z,X)\leq
\begin{cases}
\dfrac{2\pi}{e}(\cosh\gamma +1) \bigg(\dfrac{1}{\cosh \gamma -1}+ \dfrac{\pi^2}{\sigma^4}\bigg), & 0<\gamma \leq \mathrm{arccosh}\bigg(\dfrac{e+1}{e-1}\bigg),\\
2\pi(\cosh\gamma -1)\log\bigg(\dfrac{\cosh\gamma -1}{\cosh \gamma +1}\bigg)\bigg(\dfrac{1}{\cosh \gamma -1}+ \dfrac{\pi^2}{\sigma^4}\bigg),& \mathrm{otherwise}.
\end{cases}
\]
As $\gamma \rightarrow \infty$ and $\sigma^2 \rightarrow \infty$, we see that the Wasserstein metric tends to 0 which is something that one would expect, since both densities of the wrapped normal and Cauchy both tend towards the circular uniform density in these limits.

\end{example}

\section{Discussion}

To conclude, we provide a short discussion on the drawbacks of this particular approach as well as future directions of research.

As previously mentioned in Example \ref{exmp:bayes}, we were not able to obtain a bound with optimal order of $n$. This is a consequence of the route with which we took when bounding the expectation $\E[|\alpha(X)\pi_0'(X)\tau^c(X)|]$ since $\alpha$ is an intractable function. Moreover, integrals involving $\tau^c$ will also further complicate calculations. This is clearly demonstrated in Examples \ref{exmp: vm bing} and \ref{exmp:wn} where our bound is of $O(1)$. However, carefully constructed examples like Examples \ref{exmp:bayes} and \ref{exmp:wc} do show that convergence of the Wasserstein metric to 0 is possible. Computational methods could be utilised instead of brute force bounds to calculate $\alpha$ and compute the relevant expectation.

Lastly, a research area of potential interest to expand the Stein's method on $\S^1$ is the use of exchangeable pairs in order to bound the Wasserstein metric. Meckes in \cite{meckes2009approximate} discusses a method which generates exchangeable pairs on general manifold for uniform random variables. This utilises the idea of geodesic flow on the tangent bundle $TM$ of manifold $M$. The geodesic flow gives rise to a 
to a set of PDEs which one can solve to generate an exchangeable pair. This construction can also be thought of as Newtonian dynamics of a free particle on $M$ with a potential of $\log p$ where $p$ is the probability density. In the uniform case, this potential term is constant. But for the von-Mises distribution, we end up with a term that resembles the potential term in the harmonic motion of a pendulum.


\textbf{Acknowledgements:} I am eternally grateful to my supervisors Huling Le, Karthik Bharath and Chris Fallaize for the help that they have provided while writing this paper and throughout my PhD. I would also like to thank the EPSRC for funding my PhD research project.

\end{document}